\documentclass[11pt,a4paper]{article}

\usepackage[T1]{fontenc}
\usepackage{microtype}

\usepackage{amsmath}								
\usepackage{amssymb}
\usepackage{amsthm}
\usepackage{amscd}
\usepackage{amsfonts}
\usepackage{stmaryrd}

\usepackage{euler}

\usepackage{extarrows}

\usepackage[colorlinks, linktocpage, citecolor = blue, linkcolor = blue]{hyperref}
\usepackage{color}

\usepackage{tikz}									
\usetikzlibrary{matrix}
\usetikzlibrary{patterns}
\usetikzlibrary{matrix}
\usetikzlibrary{positioning}
\usetikzlibrary{decorations.pathmorphing}
\usetikzlibrary{cd}

\usepackage{fullpage}
\usepackage{enumerate}
\usepackage{subcaption}

\newtheorem*{theorem*}{Theorem}

\newtheorem{theorem}{Theorem}[section]
\newtheorem{lemma}[theorem]{Lemma}
\newtheorem{proposition}[theorem]{Proposition}
\newtheorem{corollary}[theorem]{Corollary} 

\theoremstyle{definition}
\newtheorem{definition}[theorem]{Definition}
\newtheorem{example}[theorem]{Example}

\newtheorem{remark}[theorem]{Remark}

\newcommand{\ZZ}{\mathbb{Z}}

\newcommand{\XX}{\mathbb{X}}
\newcommand{\YY}{\mathbb{Y}}

\DeclareMathOperator{\Aut}{Aut}

\DeclareMathOperator{\val}{val}

\DeclareMathOperator{\tr}{tr}

\usepackage{color}

\newcommand{\tg}{\widetilde{g}}

\renewcommand{\oe}{\overline{e}}
\newcommand{\oh}{\overline{h}}
\newcommand{\fp}{\mathfrak{p}}
\newcommand{\fq}{\mathfrak{q}}

\title{Zeta functions of edge-free quotients of graphs}

\date{}

\author{Dmitry Zakharov}

\begin{document}

\maketitle

\begin{abstract}
We consider the Ihara zeta function $\zeta(u,X/\!\!/G)$ and Artin--Ihara $L$-function of the quotient graph of groups $X/\!\!/G$, where $G$ is a group acting on a finite graph $X$ with trivial edge stabilizers. We determine the relationship between the primes of $X$ and $X/\!\!/G$ and show that $X\to X/\!\!/G$ can be naturally viewed as an unramified Galois covering of graphs of groups. We show that the $L$-function of $X/\!\!/G$ evaluated at the regular representation is equal to $\zeta(u,X)$ and that $\zeta(u,X/\!\!/G)$ divides $\zeta(u,X)$. We derive two-term and three-term determinant formulas for the zeta and $L$-functions, and compute several examples of $L$-functions of edge-free quotients of the tetrahedron graph $K_4$.

\end{abstract}

\setcounter{tocdepth}{2}



\section{Introduction}

In~\cite{1966Ihara} Ihara associated a zeta function to a discrete torsion-free cocompact subgroup $G$ of $PGL_2(K)$, where $K$ is a non-Archimedean field with residue field $\mathbb{F}_q$, and proved that this function is the reciprocal of a polynomial. Ihara's zeta function is a $p$-adic analogue of the Selberg zeta function of a Riemann surface. It was observed by Serre (see Introduction of~\cite{1980Serre}) that Ihara's zeta function can be reconstructed from the finite $(q+1)$-regular graph $X/G$, where $X$ is the Bruhat--Tits tree of $PGL_2(K)$. Building on the work of Hashimoto~\cite{1989Hashimoto}, Bass in~\cite{1992Bass} considered the more general case of a group $G$ acting without edge inversion on a locally finite tree $X$, with finite quotient graph $X/G$ and finite stabilizers. Using the technique of non-commutative determinants, Bass associated a zeta function to the quotient $X/G$, and proved that it is the reciprocal of a polynomial.

Stark and Terras~\cite{1996StarkTerras} reformulate the theory of zeta functions of graphs, and reproved the results of Bass, in a purely elementary manner. The zeta function $\zeta(u,X)$ of a graph $X$ is defined as an Euler product over the primes of $X$, which are defined as equivalence classes of closed reduced primitive paths. Stark and Terras also introduce a number of multivariable generalizations of the zeta function that can, for example, account for arbitrary edge lengths. The zeta function $\zeta(u,X)$ can also be viewed as a graph-theoretic analogue of the Dedekind zeta function $\zeta(u,K)$ of a number field $K$, and satisfies a number of analogous properties, such as a class number formula~\cite{1998Northshield} and, in some cases, a Riemann hypothesis (see~\cite{1996StarkTerras}). 

In~\cite{2000StarkTerras}, Stark and Terras further develop the analogy between graph theory and number theory, by considering zeta functions of free Galois coverings of graphs. Given a group $G$ acting on a finite graph $Y$ without edge inversion, we can form the quotient graph $X=Y/G$ and compare the zeta functions $\zeta(u,Y)$ and $\zeta(u,X)$. Stark and Terras show that if the action of $G$ is free, then $\zeta(u,Y)$ is divisible by $\zeta(u,X)$. Specifically, for such a covering $Y/X$ they construct an $L$-function $L(u,\rho,Y/X)$ valued at a representation $\rho$ of $G$, analogous to the Artin $L$-function in number theory. Evaluating the $L$-function at respectively the right regular and trivial representations of $G$ recovers $\zeta(u,Y)$ and $\zeta(u,X)$, which proves the divisibility statement. This result is the graph-theoretic analogue of the Aramata--Brauer theorem, which states that $\zeta(u,L)/\zeta(u,K)$ is an entire function when $L/K$ is a Galois extension of a number field. In fact, $\zeta(u,Y)$ is divisible by $\zeta(u,X)$ even when $Y\to X$ is a non-normal covering of graphs, while the corresponding statement for number fields, known as Dirichlet's conjecture, remains open. 

In the monograph~\cite{2011Terras}, Terras observes that if $X=Y/G$ is the quotient under a non-free action, then $\zeta(u,X)$ does not necessarily divide $\zeta(u,Y)$, and poses the problem (see pp.~115-116 and~228 in {\it loc. cit.}) of generalizing of the results of~\cite{2000StarkTerras} to non-free quotients. In~\cite{2010MalmskogManes} Malmskog and Manes consider a group action with a single nontrivial vertex stabilizer, and obtain partial divisibility results for the corresponding zeta functions. However, to the best of the author's knowledge, this question has not been properly addressed. 

The purpose of this paper is to answer Terras's question. We consider a broader class of group actions on graphs, allowing edge inversions and nontrivial vertex stabilizers, but requiring the edge stabilizers to be trivial. Given such an action of a group $G$ on a finite graph $Y$, we form the quotient graph $X=Y/G$ by allowing legs as images of inverted edges (this procedure can technically be avoided by first subdividing the edges of $Y$ at the midpoints, but this increases the sizes of the matrices involved and complicates the calculations). We then decorate the vertices $v$ of $X$ with the stabilizer groups $X_v$ of suitably chosen preimages in $Y$. The resulting object $\XX=(X,X_v)=Y/\!\!/G$ is an example of a graph of groups with trivial edge groups.

It turns out that, to generalize the divisibility result of Stark and Terras, it is necessary to replace $\zeta(u,X)$ of the graph $X$ with an Ihara zeta function $\zeta(u,\XX)$ enumerating the primes $\fq$ in the quotient graph of groups $\XX$:
$$
\zeta(u,\XX)=\prod_{\fq}(1-u^{\ell(\fq)})^{-1}.
$$
More generally, we define an Artin--Ihara $L$-function $L(u,\rho,Y/\XX)$ valued at a representation $\rho$ of $G$ (see Def.~\ref{def:L}). We derive determinant formulas for $\zeta(u,\XX)$ and $L(u,\rho,Y/\XX)$ in terms of the adjacency and valency matrices of $X$, and the vertex stabilizers (see Thm.~\ref{thm:3term} and Thm.~\ref{thm:threetermgogL}). The function $\zeta(u,\XX)$ can also be viewed as the zeta function of a graph $X$ with vertex weights specified by the degrees of the stabilizers, and may be of independent interest; a similar function was recently defined in~\cite{2019KonnoMitsusashiMoritaSato}.

In addition, we determine the relationship between the primes of $Y$ and $\XX$ (see Thm.~\ref{thm:primes}), which turns out to be identical to the case of free Galois coverings considered in~\cite{2000StarkTerras}, and in Cor.~\ref{cor:divisibility} prove that $\zeta(u,Y)$ is divisible by $\zeta(u,\XX)$. We illustrate our results with several examples of edge-free quotients of the tetrahedron graph $K_4$, which has no nontrivial free quotients.

We note that zeta functions of graphs of groups are already present in Bass's original paper~\cite{1992Bass}, since Bass considers non-free actions. Our purpose, however, is to give an elementary, self-contained treatment of the subject, in the spirit of Stark and Terras's papers. We restrict our attention to edge-free actions purely for the sake of simplicity, and our results naturally generalize to arbitrary quotients. In a future paper, we also plan to extend our results to the various multivariable generalizations of the Ihara and Bartholdi zeta functions of a graph.

\section{Preliminaries}

In this section, we recall the basic definitions and results concerning graphs and their zeta functions, group actions on graphs, and $L$-functions of Galois coverings of graphs. The only novelty is that we consider graphs with legs. Such graphs arise in tropical geometry as the underlying combinatorial objects of tropical curves with marked points. We will see that graphs with legs are also natural from the viewpoint of graph quotients and zeta functions, and allow us to consider a broader class of group actions on graphs. 

\subsection{Graphs and their zeta functions}

A {\it graph with legs} $X$, or simply a {\it graph}, consists of the following data:

\begin{itemize} \item A set of {\it vertices} $V(X)$.

\item A set of {\it half-edges} $H(X)$.

\item A {\it root map} $r:H(X)\to V(X)$.

\item An involution $H(X)\to H(X)$ denoted $h\mapsto \oh$.

\end{itemize}

The {\it valency} of a vertex $v\in V(X)$ is $\val(v)=\#(r^{-1}(v))$. An orbit $\{h,\oh\}$ of size 2 under the involution is called an {\it edge} of $X$, and the set of edges is denoted $E(X)$. An {\it oriented edge} is an ordered pair $e=(h,\oh)$, and we call $i(e)=r(h)$ and $t(e)=r(\oh)$ respectively the {\it initial} and {\it terminal} vertices of $e$. We denote $\oe=(\oh,h)$ the edge $e$ with the opposite orientation. An edge $e=\{h,\oh\}$ is called a {\it loop} if $r(h)=r(\oh)$. A half-edge fixed by the involution is called a {\it leg} of $X$, and the set of legs is denoted $L(X)$. 

A {\it path} $P$ in $X$ of {\it length} $n=\ell(P)$ is a sequence of half-edges
$$
P=h_1h_2\cdots h_n
$$
such that $r(h_{j+1})=r(\oh_j)$ for $j=1,\ldots,n-1$. The vertices
$$
v_0=r(h_1), \, v_1=r(h_2)=r(\oh_1),\,\ldots,\,v_{n-1}=r(h_n)=r(\oh_{n-1}),\,v_n=r(\oh_n)
$$
form the {\it vertex sequence} of $P$. We call $v_0$ and $v_n$ respectively the {\it initial} and {\it terminal} vertices of $P$, and we say that $P$ is {\it closed} if $v_0=v_n$. If the terminal vertex of a path $P$ coincides with the initial vertex of a path $Q$, we can form the composite path $PQ$, in particular, we can take positive integer powers of any closed path. We consider only connected graphs. 

Each half-edge $h_j$ of $P$ is either a leg or the first half-edge of an oriented edge $e_j=(h_j,\oh_j)$. If the graph $X$ has no legs, then a path $P=h_1\cdots h_n$ is a sequence $e_1\cdots e_n$ of oriented edges $e_j=(h_j,\oh_j)$ satisfying $i(e_{j+1})=t(e_j)$. If $X$ has legs, then a path in $X$ consists of a sequence of oriented edges, possibly with legs attached at the intermediate vertices, or may even consist of a sequence of legs all attached at the same vertex. 

\begin{remark} We are primarily interested in cyclically ordered closed paths of a given length $n$, so for an index $j\in \{1,\ldots,n\}$ and an integer $n$ we usually interpret $j+n$ as the corresponding residue modulo $n$. 

\label{rem:indexing}

\end{remark}

A path $P=h_1\cdots h_n$ is called {\it reduced} if $h_{j+1}\neq \oh_j$ for all $j=1,\ldots,n$. In other words, a path is reduced if it does not contain a {\it backtrack} (two adjacent oriented edges $e_j$ and $e_{j+1}$ such that $e_{j+1}=\oe_j$) or a {\it tail} (initial and final oriented edges $e_1$ and $e_n$ such that $e_1=\oe_n$), and if no leg occurs twice in a row or as both the first and the last half-edge. In particular, a single leg does not constitute a closed reduced path. A closed reduced path $P$ is called {\it primitive} if there does not exist a path $Q$ such that $P=Q^k$ for some integer $k\geq 2$. Two closed paths $P=h_1\cdots h_n$ and $P'=h'_1\cdots h'_n$ of the same length are {\it equivalent} if they differ by a choice of initial vertex, in other words if there exists an integer $0\leq m\leq n-1$ such that $h'_j=h_{j+m}$ for all $j\in \{1,\ldots,n\}$.

A {\it prime} $\fp$ of $X$ is an equivalence class of primitive paths and has a well-defined length $\ell(\fp)$. We note that a primitive path and the same path traversed in the opposite direction represent different primes of $X$. The {\it Ihara zeta function} $\zeta(u,X)$ is defined as the product
$$
\zeta(u,X)=\prod_{\fp} (1-u^{\ell(\fp)})^{-1}
$$
taken over all primes of $X$, where $u$ is a complex variable. This product is usually infinite, and converges for sufficiently small $u$. 

The Ihara {\it three-term determinant formula} (see~\cite{1992Bass}, \cite{1996StarkTerras}, or~\cite{2011Terras}) expresses the reciprocal of the zeta function $\zeta(u,X)$ of a graph $X$ without legs as an explicit polynomial in $u$. Let $V(X)=\{v_1,\ldots,v_n\}$ be the vertices of $X$, and let $m$ be the number of edges, so that $\#(H(X))=2m$. Let $A$ and $Q$ be the $n\times n$ {\it adjacency} and {\it valency} matrices of $X$, defined as
$$
A_{jk}=\#\{e\in E(X):i(e)=v_j,t(e)=v_k\},\quad Q_{jk}=\delta_{jk}\val v_j.
$$
Let $b_1(X)=m-n+1$ be the first Betti number of $X$. Then 
\begin{equation}
\zeta(u,X)^{-1}=(1-u^2)^{b_1(X)-1}\det (I_n-Au+(Q-I_n)u^2).
\label{eq:3term}
\end{equation}

We consider several examples.

\begin{example} Let $F_m$ be the graph consisting of a single vertex and $m$ loops. Then $b_1=m$, $A=(2m)$, $Q=(2m)$, therefore
$$
\zeta(u,F_m)^{-1}=(1-u^2)^{m-1}(1-2mu+(2m-1)u^2)=(1-u^2)^{m-1}(1-u)(1-(2m-1)u).
$$
\label{ex:Fm}
\end{example}

\begin{example} Let $L_m$ be the graph consisting of two vertices joined by $m$ edges. Then
$$
b_1=m-1,\quad A=\left[\begin{array}{cc} 0 & m \\ m & 0 \end{array}\right],\quad 
Q=\left[\begin{array}{cc} m & 0 \\ 0 & m \end{array}\right],
$$
therefore
$$
\zeta(u,L_m)^{-1}=(1-u^2)^{m-2}\left|\begin{array}{cc} 1 + (m-1)u^2 & -mu \\ -mu & 1 + (m-1)u^2 \end{array}\right|=
(1-u^2)^{m-1}(1-(m-1)^2u^2).
$$
\label{ex:Lm}
\end{example}

\begin{example} For our last example, we work out by hand the zeta function of the graph $X$ consisting of a single vertex and two legs $l_1$ and $l_2$. A path of length $n$ in $X$ is a sequence $l_{j_1}l_{j_2}\cdots l_{j_n}$, with $j_k\in \{1,2\}$. Such a path is always closed, and is reduced if $j_k\neq j_{k+1}$ for all $k$. 
Therefore the only reduced paths in $X$ have even length and are of the form $l_1l_2\cdots l_1l_2$ and $l_2l_1\cdots l_2l_1$. Hence $X$ has two primitive paths $l_1l_2$ and $l_2l_1$ representing a single prime of length 2, and thus
$$
\zeta(u,X)^{-1}=1-u^2.
$$
\label{ex:2legs}
\end{example}

\subsection{Group actions on graphs}

Let $X$ be a graph. An {\it automorphism} $f$ of $X$ is a pair of bijections $V(X)\to V(X)$ and $H(X)\to H(X)$ that are equivariant with respect to the root map and the involution. An automorphism maps edges to edges and legs to legs. An {\it action} of a group $G$ on a graph $X$ is a homomorphism from $G$ to the automorphism group of $X$. 

We are primarily interested in two kinds of group actions on a graph. We say that a $G$-action on $X$ is {\it edge-free} if $G$ acts freely on the set of half-edges $H(X)$. We say that a $G$-action is {\it free} if $G$ acts freely on $V(X)$ and $H(X)$, and if in addition $G$ does not flip edges, in other words if $gh\neq \oh$ for any edge $e=\{h,\oh\}$ and any $g\in G$. We note that this condition is necessary to define a graph quotient $X/G$ when using only edges, and is often part of the definition of a group action on a graph. Our use of legs allows us to relax this condition. 

We note that free actions are called {\it unramified} in~\cite{2011Terras}. We avoid this terminology, because this term has an entirely different meaning in tropical geometry, and because, as we will see in Thm.~\ref{thm:primes}, the behavior of primes in quotients by edge-free actions qualifies the latter as unramified coverings. We also note that edge-free actions are called {\it harmonic} in~\cite{2015Corry} (see Def.~6 and Prop.~7), another term that we avoid because of its use in tropical geometry.

\begin{definition} Given an action of $G$ on $X$, we form the {\it quotient graph} $X/G$ by setting $V(X/G)=V(X)/ G$ and $H(X/G)=H(X)/ G$ and descending the root and involution maps. We denote by $\pi:X\to X/G$ the natural quotient map.
\label{def:quotient}
\end{definition}
Working out the details, we see that:

\begin{itemize} \item If $h=\oh$ is a leg of $X$, then $\pi(h)$ is a leg of $X/G$, rooted at $\pi(r(h))$.

\item If $e=\{h,\oh\}$ is an edge of $X$, then there are two possibilities. There may exist an element $g\in G$ (unique of order 2 if the action is edge-free) flipping $e$, so that $g(h)=\oh$. In this case $e$ maps to a leg $\pi(h)=\pi(\oh)$ of $X/G$ (note that in this case $g(r(h))=g(r(\oh))$ and hence $\pi(r(h))=\pi(r(\oh))$). If such an element does not exist, then $\pi(h)\neq \pi(\oh)$ and $\pi(e)=\{\pi(h),\pi(\oh)\}$ is an edge of $X/G$.

\end{itemize}

We call the map $\pi:X\to X/G$ respectively a {\it free Galois covering} or {\it edge-free Galois covering} if the action of $G$ is free or edge-free. A free Galois covering of graphs is a Galois covering in the topological sense, and vice versa.

\begin{example} To illustrate our definition of a quotient, consider the graph $L_2$ consisting of two vertices joined by two edges:
\begin{center}
\begin{tikzpicture}
\draw [ultra thick] (0,0) .. controls (0.5,0.3) and (1,0.3) .. (1.5,0);
\draw [ultra thick] (0,0) .. controls (0.5,-0.3) and (1,-0.3) .. (1.5,0);
\draw[fill](0,0) circle(1mm);
\draw[fill](1.5,0) circle(1mm);
\end{tikzpicture}
\end{center}
The automorphism group of $L_2$ is the Klein $4$-group $\Aut(L_2)=\{e,g_1,g_2,g_1g_2\}$, where $g_1$, $g_2$, and $g_1g_2$ act respectively by a rotation by $\pi$, a reflection about the vertical axis, and a reflection about the horizontal axis. The action of $\Aut(L_2)$ and of all of its subgroups on $L_2$ is edge-free. The following diagram shows the quotient graphs $L_2/G$ for the four nontrivial subgroups $G\subset \Aut(L_2)$:
\begin{center}
\begin{tikzcd}
&&
\begin{tikzpicture}
\draw [ultra thick] (0,0) -- (1.5,0);
\draw[fill](0,0) circle(1mm) node[left]{$\ZZ/2\ZZ$};
\draw[fill](1.5,0) circle(1mm) node[right]{$\ZZ/2\ZZ$};
\end{tikzpicture}
&& \\
\begin{tikzpicture}
\draw [ultra thick] (0,0) .. controls (0.5,0.3) and (0.6,0.1) .. (0.6,0);
\draw [ultra thick] (0,0) .. controls (0.5,-0.3) and (0.6,-0.1) .. (0.6,0);
\draw[fill](0,0) circle(1mm);
\end{tikzpicture} 
&& 
\begin{tikzpicture}
\draw [ultra thick] (0,0) .. controls (0.5,0.3) and (1,0.3) .. (1.5,0);
\draw [ultra thick] (0,0) .. controls (0.5,-0.3) and (1,-0.3) .. (1.5,0);
\draw[fill](0,0) circle(1mm);
\draw[fill](1.5,0) circle(1mm);
\end{tikzpicture}
\arrow{ll}[swap]{\{e,g_1\}} \arrow{d}{\Aut(L_2)} \arrow{rr}{\{e,g_2\}} \arrow{u}[swap]{\{e,g_1g_2\}}
&&
\begin{tikzpicture}
\draw [ultra thick] (0,0) -- (-0.7,0);
\draw [ultra thick] (0,0) -- (0.7,0);
\draw[fill](0,0) circle(1mm);
\end{tikzpicture}
 \\
&&
\begin{tikzpicture}
\draw [ultra thick] (0,0) -- (0.7,0);
\draw[fill](0,0) circle(1mm) node[left]{$\ZZ/2\ZZ$};
\end{tikzpicture}
&& 
\end{tikzcd}
\end{center}
In this diagram, vertices are marked by a bold dot, so a line segment with no dot at the end represents a leg. We mark the nontrivial stabilizer groups of the vertices. 

As a preview of our results, we compare the zeta functions of all graphs in this diagram. The zeta function on $L_2$ is computed in Ex.~\ref{ex:Lm} and is equal to $\zeta(u,L_2)^{-1}=(1-u^2)^2$. The left quotient is $F_1$ and has $\zeta(u,F_1)^{-1}=(1-u)^2$ by Ex.~\ref{ex:Fm}. The right quotient is the graph of Ex.~\ref{ex:2legs} and the reciprocal of its zeta function is $1-u^2$. Finally, the reciprocals of the zeta functions of the top and bottom quotients, viewed as graphs of groups, are computed in Ex.~\ref{ex:2examples} and are equal to $1-u^2$ and $1-u$, respectively. We observe that the zeta functions of the quotients all divide the zeta function of $L_2$, which is an example of Cor.~\ref{cor:divisibility}.

\end{example}

\subsection{Artin $L$-function of a free Galois covering of graphs}

Finally, we review Artin $L$-functions of free Galois coverings (see~\cite{2000StarkTerras} or Part IV of~\cite{2011Terras}). Let $\pi:Y\to X$ be a free Galois covering of graphs without legs with Galois group $G$ of degree $d=\#(G)$. Let $\fp$ be a prime of $Y$ with representative $P$. The projection $\pi(P)$ is a closed path in $X$, hence $\pi(P)=Q^f$ for a unique primitive path $Q$. The equivalence class of $Q$ does not depend on the choice of $P$, and we say that $\fp$ {\it lies over} $\fq=[Q]$ with {\it residual degree} $f=f(\fp,X/Y)$. The Galois group $G$ acts transitively on the primes $\fp_1,\ldots,\fp_g$ of $Y$ lying over a given prime $\fq$ of $X$, these primes have the same residual degree, denoted $f(\fq,Y/X)$, and the number of primes over $\fq$ is $g=g(\fq,Y/X)=d/f(\fq,Y/X)$. 

Let $\fq$ be a prime of $X$, and let $Q$ be a primitive path in $X$ representing $\fq$ and starting at a vertex $u\in V(X)$. Choose a vertex $v\in V(Y)$ so that $\pi(v)=u$, then the path $Q$ lifts to a unique (not necessarily closed) path $P$ in $Y$ with initial vertex $v$ and terminal vertex lying over $u$. The {\it Frobenius automorphism} $F(Q,Y/X)\in G$ is the unique element of $G$ mapping $v$ to the terminal vertex of $P$. If $Q'$ is a different representative of $\fq$, then $F(Q,Y/X)$ and $F(Q',Y/X)$ are conjugate.

We now define the Artin--Ihara $L$-function of a free Galois covering of graphs. 

\begin{definition} Let $\pi:Y\to X$ be a free Galois covering with Galois group $G$, and let $\rho$ be a representation of $G$. We define the {\it Artin--Ihara $L$-function} as the product
$$
L(u,\rho,Y/X)=\prod_{\fq} \det(1-\rho(F(Q,Y/X))u^{\ell(\fp)})^{-1}
$$
taken over the primes $\fq$ of $X$, where $Q$ is any representative of $\fq$.
\label{def:L}
\end{definition}

The Frobenius automorphism $F(Q,Y/X)$ depends on the choice of $Q$ only up to conjugacy, hence the determinant is well-defined. As with the zeta function, this product converges for sufficiently small $u$. 

Let $\rho_1$ and $\rho_2$ be representations of $G$, then it is clear that
$$
L(u,\rho_1\oplus\rho_2,Y/X)=L(u,\rho_1,Y/X)L(u,\rho_2,Y/X).
$$
Furthermore, it is clear that the $L$-function evaluated at the trivial representation $1_G$ recovers the zeta function of $X$:
$$
L(u,1_G,Y/X)=\zeta(u,X).
$$
A deeper result is that the $L$-function evaluated at the right regular representation $\rho_G$ is the zeta function of $Y$:
\begin{equation}
L(u,\rho_G,Y/X)=\zeta(u,Y).
\label{eq:Linduction}
\end{equation}
Since the trivial representation is a summand in the regular representation, it follows that $\zeta(u,X)$ divides $\zeta(u,Y)$, which is in fact true for non-normal free coverings $Y\to X$ as well (see. Prop.~13.10 in~\cite{2011Terras}).

\section{The zeta function of an edge-trivial graph of groups}

In this section, we define the zeta function of the quotient graph of groups of an edge-free group action on a finite graph. Our primary reference for graphs of groups is~\cite{1993Bass}. By definition, the edge stabilizers and therefore the Bass--Serre relations are trivial, hence we do not deploy Bass--Serre theory in its full generality. On the other hand, we consider the more general situation of graphs with legs, so we are not able to simply cite the relevant results of~\cite{1993Bass}. 

\subsection{Definitions}

\begin{definition} An {\it edge-trivial graph of groups} $\XX$ consists of a graph $X$ and groups $X_v$ for all vertices $v\in V(X)$. We say that $\XX$ is {\it finite} if $X$ is a finite graph, and each $X_v$ is a finite group. 
\end{definition}

We now define paths, closed paths, primes, and the zeta function of a finite edge-trivial graph of groups $\XX$. A {\it path} $P$ in $\XX$ of {\it length} $n=\ell(P)$ is a sequence 
$$
P=g_0h_1g_1h_2\cdots g_{n-1}h_n
$$
where $h_1\cdots h_n$ is a path in $X$ with vertex sequence $v_0=r(h_1)$, $v_j=r(h_{j+1})=r(\oh_j)$ for $j=1,\ldots,n-1$, and $v_n=r(\oh_n)$, and $g_j$ are elements of $X_{v_j}$ for $j=0,\ldots,n-1$. We say that $P$ is {\it closed} if $v_n=v_0$, in other words if the underlying path in $X$ is closed. If the terminal vertex of a path $P$ is equal to the initial vertex of $Q$, we can define the path $PQ$ by concatenation, and we can define positive integer powers of closed paths.

We say that a path $P=g_0h_1g_1\cdots g_{n-1}h_n$ is {\it reduced} if for all $j\in \{1,\ldots,n\}$ either $h_{j+1}\neq \oh_j$ or the element $g_j$ is not the identity. A closed reduced path $P$ is called {\it primitive} if there does not exist a path $Q$ such that $P=Q^k$ for some $k\geq 2$, and two closed paths $P=g_0h_1\cdots g_{n-1} h_n$ and $P'=g'_0h'_1\cdots g'_{n-1}h'_n$ of the same length are {\it equivalent} if $g'_{j-1}=g_{j-1+m}$ and $h'_j=h_{j+m}$ for some $0\leq m\leq n-1$ and all $j\in \{1,\ldots,n\}$. 

A {\it prime} $\fp$ of $\XX$ is an equivalence class of primitive paths and has a well-defined length $\ell(\fp)$. The length does not depend on the group elements, but if $\XX$ is finite, then there are finitely many primes of a given length. We can therefore define the zeta function of a finite graph of groups $\XX$.

\begin{definition} Let $\XX$ be a finite edge-trivial graph of groups. The {\it Ihara zeta function} $\zeta(u,\XX)$ is the product
$$
\zeta(u,\XX)=\prod_{\fp}(1-u^{\ell(\fp)})^{-1}
$$
over all primes $\fp$ of $\XX$.
\label{def:zeta}
\end{definition}

As with graphs, this product is usually infinite, and converges for sufficiently small $u$, and is in fact the reciprocal of a polynomial in $u$ (see Thm.~\ref{thm:3term}). 

We observe that $\zeta(u,\XX)$ only depends on the orders of the groups $X_v$ and not on their structure, so it is not clear how this zeta function may be related to the zeta function enumerating the normal subgroups of a group. 

\begin{remark} Consider a graph of groups $\XX=(X,X_v)$ having a leg $l$ rooted at $v\in V(X)$. Let $\YY$ be the graph of groups obtained from $\XX$ by replacing $l$ with a directed edge $e$ starting at $v$ and ending at a new vertex $u$, with an order two group $X_u=\{e,g\}$ at $u$. Then there is a natural bijection between the closed reduced paths of $\XX$ and of $\YY$, which consists in replacing each instance of $l$ by $eg\oe$. This operation also preserves path length, if we stipulate that the edge $e$ has length $1/2$. Hence a leg in a graph of groups can be viewed as an extremal edge of length $1/2$ having a stabilizer of order $2$ at the extremal vertex. 

\end{remark}

Before we proceed, we calculate two simple examples by hand. 

\begin{example} Let $\XX$ be the graph of groups consisting of a single vertex with a group $\{1,g\}$ of order $2$, and a leg $l$. A path in $\XX$ of length $n$ has the form $g_0lg_1l\cdots g_{n-1}l$, with $g_j\in\{1,g\}$. Such a path is always closed, and is reduced if and only if each $g_j$ is the nontrivial element $g$. Hence $\XX$ has a unique prime $[gl]$ of length one, and
$$
\zeta(u,\XX)^{-1}=1-u.
$$
Similarly, consider the graph of groups $\YY$ consisting of two vertices with groups $\{1,g\}$ of order two at each vertex, joined by a single edge $e$. Similarly to the case above, it is clear that $\YY$ has a unique prime $[geg\oe]$ of length two, and therefore
$$
\zeta(u,\YY)^{-1}=1-u^2.
$$
We observe that $\zeta(u,\YY)$ is obtained from $\zeta(u,\XX)$ by replacing $u$ with $u^2$, which agrees with the previous remark, since $\XX$ can be viewed as $\YY$ with an edge length of $1/2$.
\label{ex:2examples}
\end{example}

Finally, we give an informal definition of the graph of groups of an edge-free quotient, which we will use in Ex.~\ref{ex:K4} in the next section.

\begin{example} Let $Y$ be a graph with an edge-free action of a group $G$, and let $X=Y/G$ be the quotient graph. We define the structure of an edge-trivial graph of groups $\XX$ on $X$ by choosing $X_v$ for $v\in V(X)$ to be the stabilizer of any preimage of $v$. In the next section we will describe exactly how we choose these stabilizer subgroups (see Def.~\ref{def:edgefreequotient}), for now we note that the order of $X_v$ does not depend on the choice of preimage of $v$, and hence $\zeta(u,\XX)$ is well-defined.
\end{example}

\subsection{Determinant formulas}

We now prove analogues of the two-term and three-term determinant formulas for the zeta function of a finite edge-trivial graph of groups.

Let $\XX=(X,X_v)$ be a finite edge-trivial graph of groups. Let $k=2m+l$ denote the number of half-edges of $X$, where $m$ and $l$ are respectively the number of edges and legs. The {\it charge} $c(v)$ of a vertex $v\in V(X)$ is the order of the group at $v$:
$$
c(v)=\# (X_v).
$$
We define the $k\times k$ {\it half-edge adjacency matrix} $W$ of $\XX$ as follows:
$$
W_{hh'}=\left\{\begin{array}{cc} 
c(r(h'))-1, & h'=\oh, \\
c(r(h')), & h'\neq \oh\mbox{ and }r(\oh)=r(h'),\\
0, & r(\oh)\neq r(h').
\end{array}\right.
$$
We observe that $W$ is related to the edge matrix of a graph (see Definition~11.2 in \cite{2011Terras}). However, our $W$ cannot be simply computed from the edge matrix by a weight specialization, since the entries $W_{hh'}$ may be nonzero when $h'=\oh$. In other words, we cannot count the primes of $\XX$ simply by looking at the primes of $X$, since a reduced path in a graph of groups may correspond to a non-reduced path in the underlying graph.

\begin{proposition} Let $\XX$ be a finite edge-trivial graph of groups. The zeta function of $\XX$ is equal to
\begin{equation}
\zeta(u,\XX)^{-1}=\det(I_k-Wu).
\label{eq:twoterm}
\end{equation}
\label{prop:twoterm}
\end{proposition}

\begin{proof} The proposition and its proof elaborate on the corresponding result for zeta functions of graphs (see Thm.~11.4 and Cor.~11.5 in~\cite{2011Terras}). The logarithm of $\zeta(u,\XX)$ is equal to a sum over all primes $\fp$ of $\XX$:
$$
\log \zeta(u,\XX)=\sum_{\fp}\sum_{j=1}^{\infty} \frac{1}{j} u^{j\ell(\fp)}.
$$
A prime $\fp$ of length $n=\ell(\fp)$ is represented by exactly $n$ distinct primitive paths, and every closed reduced path is a unique power of a primitive path, hence we can change to a sum over all primitive paths $P$, and then to a sum over all closed reduced paths $Q$:
$$
\log \zeta(u,\XX)=\sum_{P}\sum_{j=1}^{\infty}\frac{1}{j\ell(P)}u^{j\ell(P)}=
\sum_{P}\sum_{j=1}^{\infty}\frac{1}{\ell(P^j)}u^{\ell(P^j)}=
\sum_{Q}\frac{1}{\ell(Q)}u^{\ell(Q)}=\sum_{n=1}^{\infty}\frac{C_n}{n}u^n,
$$
where $C_n$ is the number of closed reduced paths in $\XX$ of length $n$.

A closed reduced path $Q$ in $\XX$ consists of a path $h_1\cdots h_n$ in $X$ (closed but not necessarily reduced) with vertex sequence $v_j=r(h_{j+1})=r(\oh_j)$, and a choice of elements $g_j\in X_{v_j}$ for $j=0,\ldots,n-1$, where $g_j$ is required to be nontrivial if $h_{j+1}=\oh_j$. Hence the number of choices of $g_j$ is equal to $W_{h_{j}h_{j+1}}$, and the total number of closed reduced paths in $\XX$ with underlying closed path $h_1\cdots h_n$ is equal to $W_{h_nh_1}W_{h_1h_2}\cdots W_{h_{n-1}h_n}$. The latter product vanishes when the path $h_1\cdots h_n$ is not closed, so summing over all possible sequences of $n$ half-edges we get that $C_n=\tr W^n$. Therefore
$$
\log \zeta(u,\XX)=\sum_{n=1}^{\infty}\frac{\tr W^n}{n}u^n=\tr\sum_{n=1}^{\infty}\frac{1}{n} (Wu)^n=
-\tr\log(I_k-Wu)=\log \det (I_k-Wu)^{-1},
$$
which completes the proof.

\end{proof}

We now derive a three-term determinant formula for $\zeta(u,\XX)$, which expresses the zeta function in terms of matrices that are defined in terms of the vertices of $\XX$. We consider the following $n\times n$ matrices, where $n=\#(V(X))$ is the number of vertices:

\begin{enumerate}

\item $A$ is the {\it adjacency matrix}, defined for a graph with legs as follows:
$$
(A)_{uv}=\#\{h\in H(X):r(h)=u,r(\oh)=v\}=\left\{\begin{array}{cc} 
\#\,\{\mbox{edges from }u\mbox{ to }v\}, & u\neq v,\\
2\cdot \#\,\{\mbox{loops at }u\}+\#\,\{\mbox{legs at }u\}, & u=v.
\end{array}
\right.
$$

\item $Q$ is the diagonal {\it valency matrix}, with entries $(Q)_{uv}=\delta_{uv} \val(v)$.

\item $C$ is the diagonal {\it charge matrix}, with entries $(C)_{uv}=\delta_{uv} c(v)$.

\end{enumerate}

\begin{theorem} Let $\XX$ be a finite edge-trivial graph of groups. The zeta function of $\XX$ is equal to
$$
\zeta(u,\XX)^{-1}=(1-u^2)^{b_1(X)-1}(1+u)^l\det (I_n-CAu+(CQ-I_n)u^2),
$$
where $b_1(X)=m-n+1$ is the first Betti number of $X$ and $l$ is the number of legs of $X$.
\label{thm:3term}
\end{theorem}

Before proceeding with the proof, we observe that $\zeta(u,\XX)$ depends only on the structure of the underlying graph $X$, which determines $b_1(X)$, $l$, $A$, and $Q$, and the orders of the groups $X_v$, and not on the group structures. Hence we can interpret the above formula as the definition of a zeta function of a graph $X$ with vertex weights, specified by the charges $c(v)$. Furthermore, there is no reason to require that these weights be integers, or even positive. A similar zeta function was recently introduced in~\cite{2019KonnoMitsusashiMoritaSato}.

\begin{proof} As with the proof of Prop.~\ref{prop:twoterm}, we generalize the corresponding proof for graphs, which is originally due to Bass (see pp.~86-90 in~\cite{2011Terras}). We order the half-edges of $X$ so that $\oh_j=h_{j+m}$ for $j=1,\ldots,m$ and $\oh_j=h_j$ for $j=2m+1,\ldots,k$, and define the following auxiliary matrices:

\begin{enumerate}

\item $S$ is the $n\times k$ start matrix, with entries $(S)_{vh}=1$ if $v=r(h)$ and zero otherwise.

\item $T$ is the $n\times k$ terminal matrix, with entries $(T)_{vh}=1$ if $v=r(\oh)$ and zero otherwise.

\item $J$ is the $k\times k$ involution matrix, with entries $(J)_{hh'}=1$ if $h'=\oh$ and zero otherwise.

\item $D$ is the $k\times k$ diagonal matrix with entries $(D)_{hh'}=\delta_{hh'} c(r(h'))$. 

\end{enumerate}

It is readily verified that these matrices satisfy the following relations:
$$
SJ=T,\quad TJ=S,\quad {}^tJ=J,\quad W+J={}^tTSD,\quad SD=CS, \quad S{}^tT=A, \quad S{}^tS=Q,
$$
whence
$$
SD{}^tT=CA,\quad SD{}^tS=CQ.
$$
A direct calculation shows that the following $(n+k)\times (n+k)$ matrices are equal:
$$
\left(\begin{array}{cc} I_n & 0 \\ {}^tT & I_k \end{array}\right)
\left(\begin{array}{cc} I_n(1-u^2) & SDu \\ 0 & I_k-({}^tTSD-J)u \end{array}\right)=
$$
$$
=\left(\begin{array}{cc} I_n-SD{}^tTu+(SD{}^tS-I_n)u^2 & SDu \\ 0 & I_k+Ju \end{array}\right)
\left(\begin{array}{cc} I_n & 0 \\ {}^tT-{}^tSu & I_k \end{array}\right).
$$
Taking determinants, we see that
$$
(1-u^2)^n\det(I_k-({}^tTSD-J)u)=\det(I_n-SD^tTu+(SD{}^tS-I_n)u^2)\det (I_k+Ju).
$$
The last determinant is evaluated using our half-edge ordering as follows:
$$
\det (I_k+Ju)=\left|\begin{array}{ccc} I_m & I_mu & 0 \\ I_mu & I_m & 0 \\ 0 & 0 & I_l(1+u)\end{array}\right|=
\left|\begin{array}{cc} I_m & I_mu  \\ I_mu & I_m \end{array}\right|(1+u)^l=(1-u^2)^m (1+u)^l,
$$
where we used the following equality to compute the $2m\times 2m$ determinant:
$$
\left(\begin{array}{cc} I_m & 0 \\ -I_mu & I_m
\end{array}\right)
\left(\begin{array}{cc} I_m & I_mu \\ I_mu & I_m
\end{array}\right)=
\left(\begin{array}{cc} I_m & I_mu \\ 0 & I_m(1-u^2)
\end{array}\right).
$$
Hence, using the two-term determinant formula~\eqref{eq:twoterm} and replacing $S$, $T$, $J$, and $D$ with $A$, $Q$, and $C$, we see that
$$
\zeta(u,\XX)^{-1}=(1-u^2)^{m-n}(1+u)^l\det(I_n-CAu+(CQ-I_n)u^2),
$$
which completes the proof.
\end{proof}

\begin{example} \label{ex:K4} We now enumerate the edge-free quotients of $K_4$, the complete graph on $4$ vertices, and compute their zeta functions. First, we compute $\zeta(u,K_4)$ using the three-term determinant formula~\eqref{eq:3term}:
$$
\zeta(u,K_4)^{-1}=(1-u^2)^2(1-u)(1-2u)(1+u+2u^2)^3.
$$
The automorphism group of $K_4$ is $S_4$, and the action of a subgroup $G\subset S_4$ on $K_4$ is edge-free if and only if $G$ does not contain a transposition. There are, up to conjugacy, five nontrivial subgroups with this property:

\begin{itemize} \item The order $2$ subgroup $C_{2,2}$ generated by a double transposition.

\item The order $3$ subgroup $C_3$ generated by a $3$-cycle. 

\item The Klein $4$-group $V_4$.

\item The order $4$ subgroup $C_4$ generated by a $4$-cycle. 

\item The alternating group $A_4$ of order $12$.
\end{itemize}

Fig.~\ref{fig:K4quotients} shows the corresponding quotients of $K_4$. Vertices are marked by a bold dot, so a line segment with no dot at the end represents a leg. We observe that $K_4$ does not admit a non-trivial free action, in other words $K_4$ is not a free Galois covering of any other graph.

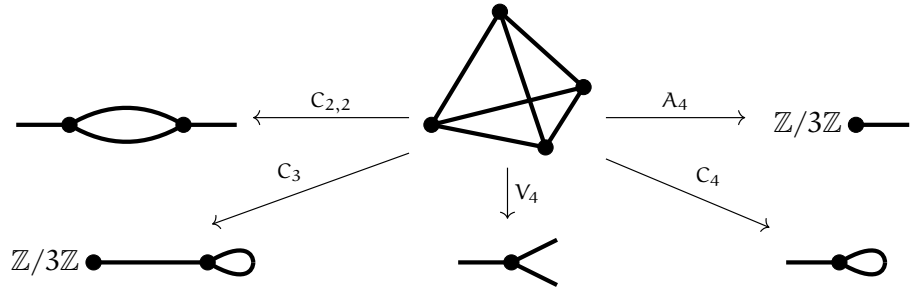
\begin{figure}[h]
\centering
\begin{tikzcd}
\begin{tikzpicture}
\draw [ultra thick] (0,0) .. controls (0.5,0.3) and (1,0.3) .. (1.5,0);
\draw [ultra thick] (0,0) .. controls (0.5,-0.3) and (1,-0.3) .. (1.5,0);
\draw[fill](0,0) circle(1mm);
\draw[fill](1.5,0) circle(1mm);
\draw [ultra thick] (0,0) -- (-0.7,0);
\draw [ultra thick] (1.5,0) -- (2.2,0);
\end{tikzpicture}
&&
\begin{tikzpicture}
\draw[fill](0,0) circle(1mm);
\draw[fill](2.0,0.5) circle(1mm);
\draw[fill](1.5,-0.3) circle(1mm);
\draw[fill](0.9,1.5) circle(1mm);
\draw [ultra thick] (0,0) -- (2.0,0.5);
\draw [ultra thick] (0,0) -- (1.5,-0.3);
\draw [ultra thick] (0,0) -- (0.9,1.5);
\draw [ultra thick] (2.0,0.5) -- (1.5,-0.3);
\draw [ultra thick] (2.0,0.5) -- (0.9,1.5);
\draw [ultra thick] (1.5,-0.3) -- (0.9,1.5);
\end{tikzpicture}
\arrow{ll}[swap]{C_{2,2}} \arrow{lld}[swap]{C_3} \arrow{d}{V_4} \arrow{rrd}{C_4} \arrow{rr}{A_4}
&& 
\begin{tikzpicture}
\draw [ultra thick] (0,0) -- (0.7,0);
\draw[fill](0,0) circle(1mm) node[left]{$\mathbb{Z}/3\mathbb{Z}$};
\end{tikzpicture}
\\
\begin{tikzpicture}
\draw [ultra thick] (0,0) -- (-1.5,0);
\draw [ultra thick] (0,0) .. controls (0.5,0.3) and (0.6,0.1) .. (0.6,0);
\draw [ultra thick] (0,0) .. controls (0.5,-0.3) and (0.6,-0.1) .. (0.6,0);
\draw[fill](0,0) circle(1mm);
\draw[fill](-1.5,0) circle(1mm) node[left]{$\mathbb{Z}/3\mathbb{Z}$};
\end{tikzpicture} 
&& 
\begin{tikzpicture}
\draw [ultra thick] (0,0) -- (-0.7,0);
\draw [ultra thick] (0,0) -- (0.6,0.3);
\draw [ultra thick] (0,0) -- (0.6,-0.3);
\draw[fill](0,0) circle(1mm);
\draw[fill,white](0,0.4) circle(0.1mm); 
\end{tikzpicture}
&& 
\begin{tikzpicture}
\draw [ultra thick] (0,0) -- (-0.7,0);
\draw [ultra thick] (0,0) .. controls (0.5,0.3) and (0.6,0.1) .. (0.6,0);
\draw [ultra thick] (0,0) .. controls (0.5,-0.3) and (0.6,-0.1) .. (0.6,0);
\draw[fill](0,0) circle(1mm);
\end{tikzpicture}
\end{tikzcd}
\caption{Edge-free quotients of $K_4$}
\label{fig:K4quotients}
\end{figure}

In Fig.~\ref{fig:K4zetas} we calculate the zeta functions of the quotient graphs of groups $K_4/G$ for various $G$. We observe that all of them divide the zeta function of $K_4$, which follows from Cor.~\ref{cor:divisibility}.

\begin{figure}[h]
\centering
\begin{tabular}{c|c|c|c|c|c|c|c}
$G$ & $K_4/G$ & $b_1$ & $l$ & $A$ & $Q$ & $C$ & $\zeta^{-1}(u,K_4/G)$ \\
 \hline
 $C_{2,2}$ &
\begin{tikzpicture}
\draw [ultra thick] (0,0) .. controls (0.5,0.3) and (1,0.3) .. (1.5,0);
\draw [ultra thick] (0,0) .. controls (0.5,-0.3) and (1,-0.3) .. (1.5,0);
\draw[fill](0,0) circle(1mm);
\draw[fill](1.5,0) circle(1mm);
\draw [ultra thick] (0,0) -- (-0.7,0);
\draw [ultra thick] (1.5,0) -- (2.2,0);
\end{tikzpicture}
& $1$ & $2$ & $\left(\begin{array}{cc} 1 & 2 \\ 2 & 1\end{array}\right)$ & $\left(\begin{array}{cc} 3 & 0 \\ 0 & 3\end{array}\right)$ & $\left(\begin{array}{cc} 1 & 0 \\ 0 & 1\end{array}\right)$ & 
\begin{tabular}{c} $(1+u)^2(1-u)\times$ \\  $(1-2u)(1+u+2u^2)$\end{tabular} \\
\hline
$C_3$ &  
\begin{tikzpicture}
\draw [ultra thick] (0,0) -- (-1.5,0);
\draw [ultra thick] (0,0) .. controls (0.5,0.3) and (0.6,0.1) .. (0.6,0);
\draw [ultra thick] (0,0) .. controls (0.5,-0.3) and (0.6,-0.1) .. (0.6,0);
\draw[fill](0,0) circle(1mm);
\draw[fill](-1.5,0) circle(1mm) node[left]{$\mathbb{Z}/3\mathbb{Z}$};
\end{tikzpicture}
&  $1$ & $0$ & $\left(\begin{array}{cc} 0 & 1 \\ 1 & 2\end{array}\right)$ & $\left(\begin{array}{cc} 1 & 0 \\ 0 & 3\end{array}\right)$ & $\left(\begin{array}{cc} 3 & 0 \\ 0 & 1\end{array}\right)$ & \begin{tabular}{c} $(1-u)(1-2u)\times $ \\ $(1+u+2u^2)$\end{tabular} \\
 \hline
 $V_4$ &
\begin{tikzpicture}
\draw [ultra thick] (0,0) -- (-0.7,0);
\draw [ultra thick] (0,0) -- (0.6,0.3);
\draw [ultra thick] (0,0) -- (0.6,-0.3);
\draw[fill](0,0) circle(1mm);
\draw[fill,white](0,0.4) circle(0.1mm); 
\end{tikzpicture}
 &  $0$ & $3$ & $(3)$ & $(3)$ & $(1)$ & $(1+u)^2(1-2u)$ \\
\hline
$C_4$ &
\begin{tikzpicture}
\draw [ultra thick] (0,0) -- (-0.7,0);
\draw [ultra thick] (0,0) .. controls (0.5,0.3) and (0.6,0.1) .. (0.6,0);
\draw [ultra thick] (0,0) .. controls (0.5,-0.3) and (0.6,-0.1) .. (0.6,0);
\draw[fill](0,0) circle(1mm);
\end{tikzpicture}
 &  $1$ & $1$ & $(3)$ & $(3)$ & $(1)$ & $(1+u)(1-u)(1-2u)$ \\
\hline
 $A_4$ &
\begin{tikzpicture}
\draw [ultra thick] (0,0) -- (0.7,0);
\draw[fill](0,0) circle(1mm) node[left]{$\mathbb{Z}/3\mathbb{Z}$};
\end{tikzpicture}
 & $0$ & $1$ & $(1)$ & $(1)$ & $(3)$ & $1-2u$ \\

\end{tabular}
\caption{Zeta functions of edge-free quotients of $K_4$}
\label{fig:K4zetas}
\end{figure}

\end{example}

\section{Path lifting, primes, and the $L$-function}

In this section, we consider an edge-free action of a group $G$ on a finite graph $Y$. We denote the quotient graph by $X=Y/G$ and the projection map by $\pi:Y\to X$. We define a quotient graph of groups $\XX=Y/\!\!/G$ and determine the relationship between the primes of $Y$ and $\XX$. Our description of the quotient $Y/\!\!/G$ and the procedure of lifting paths from $\XX$ to $Y$ is borrowed from Bass (see Ch.~3 of~\cite{1993Bass}), with some necessary modifications to include the legs of $Y$ and $X$.

\subsection{Path-lifting and primes in edge-free quotients}

We first describe the quotient $Y/\!\!/G$. Choose a spanning tree $T_X\subset X$ (not containing the legs of $X$), and choose a connected lift $T\subset Y$ of $T_X$. The map $\pi:V(T)\to V(T_X)=V(X)$ is a bijection, and for $v\in V(X)$ we denote by $v^T$ its lift to $T$.

\begin{definition} The quotient $\XX=Y/\!\!/G$ is the graph of groups with underlying graph $X=Y/G$, and whose vertex group $X_v$ at $v\in V(X)$ is the stabilizer of $v^T\in V(Y)$.
\label{def:edgefreequotient}

\end{definition}

We now describe the procedure of lifting a path from $\XX$ to $Y$. We mimic the setup of a free Galois covering by partitioning the half-edges of $Y$ into sheets indexed by elements of $G$. First, we choose the identity sheet $S_1=S$. For each half-edge $h\in H(X)$ we choose a half-edge $h^S\in H(Y)$ above it, and an element $F(h)\in G$, called the {\it Frobenius element} of $h$, in the following way. The half-edges of $X$ consist of those that are in $T_X$, the $b_1(X)$ edges that are not in the spanning tree, and the legs:
$$
H(X)=H(T_X)\cup \{e_1,\ldots,e_{b_1(X)}\}\cup L(X).
$$
For a half-edge $h\in H(T_X)$ we denote by $h^S$ its lift to $T$, and set $F(h)=1$. For each edge $e=(h,\oh)$ of $X$ not in the spanning tree, we choose an orientation and denote by $u=r(h)$ and $v=r(\oh)$ the initial and terminal vertices. Choose a half-edge $h^S\in H(Y)$ above $h$ that is rooted on $T$, so that $\pi(h^S)=h$ and $r(h^S)=u^T$. The complementary half-edge $\overline{h^S}$ is rooted at a vertex $r(\overline{h^S})\in V(Y)$ lying above the vertex $v\in V(X)$, and we choose $F(h)\in G$ to be any element such that $F(h)(v^T)=r(\overline{h^S})$. We then set $F(\oh)=F(h)^{-1}$ and $(\oh)^S=F(\oh)(\overline{h^S})\in H(Y)$; this is a half-edge lying above $\oh$ and rooted at $v^T$. 

Finally, a leg $l$ of $X$ is the quotient of either a set of legs of $Y$, or a set of edges (see Def.~\ref{def:quotient} and following discussion). In the former case, we choose a leg $l^S\in L(Y)$ over $l$ rooted at $r(l)^T$, and set $F(l)=1$. In the latter, we choose an edge $\{h,\oh\}\in E(Y)$ so that $\pi(h)=l$ and $r(h)=r(l)^T$, denote $l^S=h$, and denote by $F(l)$ the unique element of order 2 sending $h$ to $\oh$. 

We have constructed a set of half-edges $\{h^S:h\in H(X)\}$ of $Y$, containing $H(T)$, that maps bijectively to $H(X)$ and that are all rooted at vertices of $T$, so that $r(h^S)=r(h)^T$. The set $S_1=V(T)\cup \{h^S:h\in H(X)\}$ (which is not generally a subgraph of $Y$) is called the {\it identity sheet} of the covering $\pi:Y\to X$, relative to all of the choices made above. For $g\in G$, we denote the {\it $g$-th sheet} of the covering by $S_g=gS_1$. We observe that the set of half-edges of $Y$ is the disjoint union of the $H(S_g)$ over all $g\in G$, hence to any half-edge $h\in H(Y)$ we associate its {\it sheet number} $N(h)\in G$; this is the unique element of $G$ such that $N(h)(\pi(h)^S)=h$. The essential difference with free coverings, however, is that a vertex with a nontrivial stabilizer lies on multiple sheets: for any $v\in V(Y)$ we can find $g\in G$ such that $v=g(\pi(v)^T)$, and then $v\in V(S_{gg'})$ if and only if $g'\in X_{\pi(v)}$.

We also note that for any $h\in H(X)$, the Frobenius element $F(h)$ is the sheet number of $\overline{h^S}$, in other words $F(h)\oh^S=\overline{h^S}$ and therefore $N(\overline{f})=N(f)F(\pi(f))$ for any $f\in H(Y)$. In addition, $F(\oh)=F(h)^{-1}$ for all $h\in H(X)$, since $F(l)^2=1$ for any leg $l\in L(X)$. 

\begin{definition} Let $P=g_0h_1g_1h_2\cdots g_{n-1}h_n$ be a path in $\XX$. We define the {\it Frobenius element} $F(P)\in G$ of $P$ as follows:
$$
F(P)=g_0F(h_1)g_1F(h_2)\cdots g_{n-1}F(h_n).
$$

\end{definition}

The map $P\mapsto F(P)$ is a group homomorphism from the path algebra of $\XX$ (defined as the free product of the vertex groups and the free group $F(H(X))$ modulo the relations $\oh=h^{-1}$ for all $h\in H(X)$) to $G$. It is also clear that if the terminal vertex of a path $P_1$ is the initial vertex of $P_2$, then $F(P_1P_2)=F(P_1)F(P_2)$.

We are now ready to define the lifting of a path from $\XX$ to $Y$. For free coverings, lifting a path requires choosing a starting vertex of the covering. In our setting vertices do not, in general, distinguish sheets, therefore we instead specify the sheet number.

\begin{definition} Let $Q=g_0h_1g_1h_2\cdots g_{n-1}h_n$ be a path in $\XX$, and let $g\in G$. We define the path $\pi^{-1}(Q,g)$ in $Y$, called the {\it lift of $Q$ starting on the $g$-th sheet}, as
$$
\pi^{-1}(Q,g)=f_1f_2\cdots f_n,
$$
where $f_j=g\tg_j h_j^S$, and the elements $\tg_j\in G$ for $j=1,\ldots,n$ are defined as follows:
$$
\tg_1=F(g_0),\tg_2=F(g_0h_1g_1),\ldots, \tg_n=F(g_0h_1g_1\cdots h_{n-1}g_{n-1}).
$$

\end{definition}

We check that $\pi^{-1}(Q,g)$ is indeed a path. Let $v_0,\ldots,v_n$ be the vertex sequence of $Q$, so that $r(h_j)=v_{j-1}$ and $r(\oh_j)=v_j$ for $j=1,\ldots,n$. We need to check that the half-edges $f_j$ and $\overline{f_{j-1}}$ are rooted at the same vertex for $j=2,\ldots,n$. We see that $\tg_{j+1}=\tg_jF(h_j)g_j$, therefore
$$
f_{j+1}=g\tg_{j+1}h^S_{j+1}=g\tg_jF(h_j)g_jh^S_{j+1}.
$$
The half-edge inversion commutes with the group action, so
$$
\overline{f_j}=g\tg_j \overline{h_j^S}=g\tg_j F(h_j) \oh_j^S.
$$
It follows that
$$
r(\overline{f_j})=g\tg_j F(h_j) r(\oh_j^S)=g\tg_j F(h_j)r(\oh_j)^T=g\tg_j F(h_j) v_j^T=
g\tg_j F(h_j) g_jv_j^T=g\tg_j F(h_j) g_jr(h^S_{j+1})=r(f_{j+1}),
$$
where we used the fact that $g_j$ lies in the stabilizer of $v_j^T$.

We will require the following properties of path lifting, which are trivially verified.

\begin{lemma} Let $Q$ be a path in $\XX$. \label{lem:lift}

\begin{enumerate} \item For any $g,g'\in G$, the lifts of $Q$ starting on the $g$-th and $g'$-th sheets are related as follows:
$$
g^{-1}\pi^{-1}(Q,g)=(g')^{-1}\pi^{-1}(Q,g').
$$

\item If $Q$ is closed and $F(Q)=1$, then any lift $\pi^{-1}(Q,g)$ of $Q$ is closed. 

\item If $Q$ is closed, then for any $g\in G$ and any $k\geq 1$ the lift of $Q^k$ starting on the $g$-th sheet is equal to
$$
\pi^{-1}(Q^k,g)=P_0P_1\cdots P_{k-1}, \quad P_j=F(Q)^j \pi^{-1}(Q,g).
$$

\end{enumerate}

\end{lemma}

We note that although we say that $\pi^{-1}(Q,g)$ starts on the $g$-th sheet, the first half-edge of $\pi^{-1}(Q,g)$ in fact lies on the $gg_0$-th sheet (in other words, the path starts on the $g$-th sheet and immediately moves to the $gg_0$-th sheet via its first element $g_0$). However, the starting vertex of $\pi^{-1}(Q,g)$ does lie on the $g$-th sheet (as well as the $gg_0$-th). 

In general, we cannot reverse this procedure and define the image $\pi(P)$ of a path $P=f_1f_2\cdots f_n$ in $Y$. Specifically, if we set $h_j=\pi(f_j)$, then we can solve $f_j=g\tg_jh_j^S=N(f_j)h_j^S$ for $g_j$ to get $g_j=F(h_j)^{-1}N(f_j)^{-1}N(f_{j+1})$ for $j=1,\ldots,n-1$. However, solving for $g_0$ we get $g_0=g^{-1}N(f_1)$, so to define $\pi(P)$ we need to manually specify a starting sheet $g$ for $P$, which can be any sheet containing the initial vertex $r(f_1)$. However, if $P$ is a closed path, we can naturally define the {\it ending sheet number} as $N(\overline{f_n})=N(f_n)F(h_n)$, the sheet number of the involution of the last half-edge of $P$. This sheet number is the natural choice for the starting sheet number $g$, hence we can make the following definition.

\begin{definition} Let $P=f_1f_2\cdots f_n$ be a closed path in $Y$. The {\it image} of $P$ is the closed path
$$
\pi(P)=g_0h_1g_1\cdots g_{n-1}h_n
$$
where
$$
h_j=\pi(f_j), \quad g_j=F(h_j)^{-1}N(f_j)^{-1}N(f_{j+1}), \quad j=1,\ldots,n-1,\quad g_0=F(h_n)^{-1}N(f_n)^{-1}N(f_1).
$$
\end{definition}

We establish some elementary properties of the image of a closed path. 

\begin{lemma} Let $P=f_1f_2\cdots f_n$ and $P'=f'_1f'_2\cdots f'_n$ be closed paths in $Y$. \label{lem:image}

\begin{enumerate}

\item $\pi(P)=\pi(P')$ if and only if $P'=gP$ for some $g\in G$.  \label{item:1}

\item $F(\pi(P))=1$. \label{item:2}

\item $\pi^{-1}(\pi(P),N(\overline{f_n}))=P$. \label{item:3}

\item $P$ is reduced if and only if $\pi(P)$ is reduced.  \label{item:4}

\end{enumerate}

\end{lemma}

\begin{proof} For any $g\in G$ we have $\pi(gf_j)=\pi(f_j)$ and $N(gf_j)=gN(f_j)$, hence $\pi(gP)=\pi(P)$. Conversely, if $\pi(P)=\pi(P')$ then $f'_j=G_j f_j$ for some $G_j\in G$ and therefore $N(f'_j)=G_j N(f_j)$, and the condition $N(f_j)^{-1}N(f_{j+1})=N(f'_j)^{-1}N(f'_{j+1})$ then implies that $G_{j+1}=G_j$ for all $j$. This proves~\ref{item:1}, and~\ref{item:2} and~\ref{item:3} are trivially verified. To prove~\ref{item:4}, suppose that $f_{j+1}=g\tg_jF(h_j)g_jh^S_{j+1}$ is equal to $\overline{f_j}=g\tg_j F(h_j) \oh_j^S$ for some $j=1,\ldots,n-1$, then cancelling we see that $g_jh_{j+1}^S=\oh_j^S$, which is only possible if $g_j=1$ and $h_{j+1}=\oh_j$. Similarly, $f_1=\overline{f_n}$ is equivalent to $g_1=1$ and $h_1=\oh_n$.

\end{proof}

It follows that the image of a primitive path in $Y$ is a closed reduced path in $\XX$, so we can define the residual degree of a primitive path in $Y$:

\begin{definition} Let $P$ be a primitive path in $Y$. We say that $P$ {\it lies over} the primitive path $Q$ in $\XX$ if $\pi(P)=Q^f$, and we call $f=f(P,Y/\XX)$ the {\it residual degree} of $P$ relative to the covering $Y/\XX$.

\end{definition}

Finally, we consider what happens when we pass to equivalence classes. 

\begin{lemma} Let $\fp$ be a prime of $Y$, and let $P=f_1\cdots f_n$ and $P'=f'_1\cdots f'_n$ be representatives of $\fp$ with residual degrees $f$ and $f'$, so that $\pi(P)=Q^f$ and $\pi(P')=(Q')^{f'}$ for primitive paths $Q$ and $Q'$ in $\XX$. Then $f=f'$ and $Q\sim Q'$.

\end{lemma}

\begin{proof} By assumption, $f'_n=f_{n+k}$ for some integer $k$. If $Q=g_0h_1\cdots g_{n-1} h_n$, then it follows that $\pi(P')=(g_kh_{1+k}\cdots g_{n-1+k}h_{n+k})^f$,  which implies that $f'$ divides $f$. By symmetry, $f$ divides $f'$, hence they are equal, and $Q$ and $Q'$ are equivalent.

\end{proof}

We can therefore say that a prime $\fp$ of $Y$ lies over a prime $\fq$ of $\XX$ with residual degree $f=f(\fp,Y/\XX)$. We now determine the relationship between the primes of $Y$ and $\XX$, which turns out to be identical to the free case.

\begin{theorem} Let $Y$ be a finite graph, let $G$ be a group of order $d$, let $\XX$ be the quotient of $Y$ by an edge-free $G$-action, and let $\fq$ be a prime of $\XX$.

\begin{enumerate} \item The residual degree $f(\fp,Y/\XX)$ of any prime $\fp$ over $\fq$ is equal to the order $f(\fq,Y/\XX)$ of the Frobenius element $F(Q)$ in $G$, where $Q$ is any representative of $\fq$. \label{item:prime1}

\item The length of any prime $\fp$ over $\fq$ is equal to $\ell(\fp)=f(\fq,Y/\XX)\ell(\fq)$. \label{item:prime2}

\item $G$ acts transitively on the primes over $\fq$. \label{item:prime3}

\item The number of primes over $\fq$ is equal to $g=g(\fq,Y/\XX)=d/f(\fq,Y/\XX)$. \label{item:prime4}

\end{enumerate}

\label{thm:primes}

\end{theorem}

\begin{proof} We first note that if $Q=Q_1Q_2$ and $Q'=Q_2Q_1$ are two representatives of $\fq$, then $F(Q')=F(Q_1)^{-1}F(Q)F(Q_1)$ and hence the elements $F(Q)$ and $F(Q')$ have the same order, so the order $f(\fq,Y/\XX)$ is well-defined. 

Choose a representative $Q$ of $\fq$, and let $P=f_1\cdots f_n$ be a primitive path in $Y$ lying over $Q$ with residual degree $f(P,Y/\XX)$, so that $\pi(P)=Q^{f(P,Y/\XX)}$. By Lem.~\ref{lem:image} $F(Q^{f(P,Y/\XX)})=1$, hence $f(P,Y/\XX)=kf(Q,Y/\XX)$ for some $k\geq 1$. Now consider the path $P'=\pi^{-1}(Q^{f(Q,Y/\XX)},N(\overline{f_n}))$. By Lem.~\ref{lem:lift} this path is closed, and the lift of $(P')^k$ is equal to
$$
\pi^{-1}((P')^k,N(\overline{f_n}))=\pi^{-1}(P',N(\overline{f_n}))^k.
$$
It follows that $\pi((P')^k)=Q^{f(P,Y/\XX)}=\pi(P)$, so $P$ and $(P')^k$ are conjugate, and comparing the last half-edge we see that in fact $P=(P')^k$. Since $P$ is primitive we have $k=1$ and therefore $f(P,Y/\XX)=f(Q,Y/\XX)$. This proves~\ref{item:prime1}, and~\ref{item:prime2} follows immediately.

Now let $P$ and $P'$ be two primitive paths over $Q$. We have already seen that $\pi(P')=Q^{f(Q,Y/\XX)}=\pi(P)$, hence $P$ and $P'$ are conjugate by Lem.~\ref{lem:image}. Conversely, $\pi(gP)=\pi(P)=Q^{f(Q,Y/\XX)}$ for any $g\in G$, proving~\ref{item:prime3}. Finally, to prove~\ref{item:prime4}, we observe that there are exactly $d$ primitive paths in $Y$ lying over $Q$, namely the lifts $\pi^{-1}(Q^{f(Q,Y/\XX)},g)$ as $g$ ranges over $G$. Therefore, over all of the $\ell(\fq)$ representatives of $Q$ there are a total of $d\ell(\fq)$ primitive paths in $Y$, each having length $f(Q,Y/\XX)\ell(\fq)$, for a total of $d/f(Q,Y/\XX)$ distinct primes.

\end{proof}

We see that the behavior of primes in an edge-free quotient $\pi:Y\to Y/\!\!/G$ by a group $G$ of order $d$ is exactly the same as the splitting of prime ideals in an unramified Galois extension of degree $d$. To reinforce the analogy with number fields, we write
$$
\pi^*(\fq)=\fp_1\cdots \fp_g,
$$
where $\fp_1,\ldots,\fp_g$ are the primes over $\fq$.

\begin{remark} Let $Q$ be a primitive path in $\XX$, and let $k$ be the smallest positive integer such that $Q^k$ lifts to a closed path in $Y$. Then $k$ divides $f(Q,Y/\XX)$ but is not generally equal to it. For example, if $F(Q)$ is a nontrivial element of the stabilizer group of the initial vertex of $Q$, then $Q$ itself lifts to a closed path $P$ (see Ex.~\ref{ex:K4modC3} and~\ref{ex:K4modA4}). However, in this case a power $Q^j$ lifts to $P\cdot F(Q)P\cdot \cdots F(Q)^{j-1}P$, hence the primitive path of $Y$ that lies over $Q$ is $P\cdot F(Q)P\cdots F(Q)^{f-1}P$, not $P$.
\label{rem:shortestclosed}
\end{remark}

We illustrate our results by determining the splitting type of primes of small length in two edge-free quotients of $K_4$ (see Ex.~\ref{ex:K4}), the complete graph on four vertices. We label the vertices of $K_4$ by $1$, $2$, $3$, and $4$, and denote the edges $e_{ij}=(h_{ij},h_{ji})$ for $i\neq j$, so that $r(h_{ij})=i$. A closed path in $K_4$ can be specified by its vertex sequence (skipping the last vertex), so for example we denote the closed path $e_{12}e_{23}e_{31}$ by $123$. Enumerating the primes of $K_4$ of length five or less, we find that there are no primes of length one or two, 8 primes of length three of the form $abc$, 6 primes of length four of the form $abcd$, and no primes of length five. 

\begin{example} \label{ex:K4modC3} 

We first consider the subgroup $C_3\subset A_4$ generated by $g=(234)$. The quotient graph $X=A_4/C_3$ has two vertices $u$ and $v$ joined by an edge $e=\{h,\oh\}$, where $r(h)=u$ and $r(\oh)=v$, and a loop $\{h',\oh'\}$ at $v$. The stabilizer at $u$ is $C_3$, while the stabilizer at $v$ is trivial. The map $\pi:K_4\to X$ sends $1$ to $u$, the remaining vertices to $v$, $h_{1j}$ to $h$, and $h_{23}$, $h_{34}$, and $h_{41}$ to $h'$. 

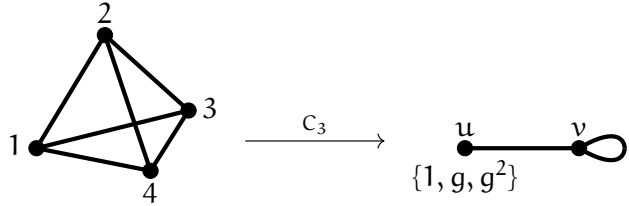
\begin{figure}[h]
\centering
\begin{tikzcd}
\begin{tikzpicture}
\draw[fill](0,0) circle(1mm) node[left]{$1$};
\draw[fill](2.0,0.5) circle(1mm) node[right]{$3$};
\draw[fill](1.5,-0.3) circle(1mm) node[below]{$4$};
\draw[fill](0.9,1.5) circle(1mm) node[above]{$2$};
\draw [ultra thick] (0,0) -- (2.0,0.5);
\draw [ultra thick] (0,0) -- (1.5,-0.3);
\draw [ultra thick] (0,0) -- (0.9,1.5);
\draw [ultra thick] (2.0,0.5) -- (1.5,-0.3);
\draw [ultra thick] (2.0,0.5) -- (0.9,1.5);
\draw [ultra thick] (1.5,-0.3) -- (0.9,1.5);
\end{tikzpicture}
\arrow{rr}{C_3} && 
\begin{tikzpicture}
\draw [ultra thick] (0,0) -- (-1.5,0);
\draw [ultra thick] (0,0) .. controls (0.5,0.3) and (0.6,0.1) .. (0.6,0);
\draw [ultra thick] (0,0) .. controls (0.5,-0.3) and (0.6,-0.1) .. (0.6,0);
\draw[fill](0,0) circle(1mm) node[above]{$v$};
\draw[fill](-1.5,0) circle(1mm) node[below]{$\{1,g,g^2\}$} node[above]{$u$};
\end{tikzpicture} 
\end{tikzcd}
\end{figure}
We lift the spanning tree $\{u,v\}\cup\{h,\oh\}$ to $T=\{1,2\}\cup \{h_{12},h_{21}\}$, so that $h^S=h_{12}$ and $\oh^S=h_{21}$, and set $h'^S=h_{23}$ and $\oh'^S=h_{24}$. With respect to these choices, the sheets are
$$
S_1=\{1,2\}\cup\{h_{12}, h_{2j}\},\quad S_g=\{1,3\}\cup\{h_{13}, h_{3j}\},\quad S_{g^2}=\{1,4\}\cup\{h_{14}, h_{4j}\},
$$
and the Frobenius elements are the following:
$$
F(h)=1,\quad F(\oh)=1,\quad F(h')=g,\quad F(\oh')=g^2.
$$ 

We now find all primes $\fq$ in $\XX$ with the property that $\ell(\fq)f(\fq)\leq 5$. The length one primes $h'$ and $\oh'$ have nontrivial Frobenius elements, hence $f(h',Y/\XX)=f(\oh',Y/\XX)=3$ and therefore they each have a unique prime over them. Indeed, we see that $(h')^3$ and $(\oh')^3$ lift to respectively $234$ and $243$, which are fixed by $C_3$. There are no primes of length two, and four primes $ghh'\oh$, $g^2hh'\oh$, $gh\oh'\oh$, and $g^2h\oh'\oh$ of length three. The prime $gh\oh'\oh$ has trivial Frobenius element, so $f(gh\oh'\oh)=1$ and the three lifts of $gh\oh'\oh$ are the conjugate primes $132$, $143$, and $124$. Similarly, the primes over $g^2hh'\oh$ are $123$, $134$, and $142$. Finally, $\XX$ has two primes $ghh'^2\oh$ and $g^2h\oh'^2\oh$ of length four that have trivial Frobenius element, accounting for the six primes of $Y$ of length four, and no primes of length five have trivial Frobenius element. We list the prime splittings below:
$$
\pi^*(h')=234,\quad \pi^*(\oh')=243,\quad \pi^*(gh\oh'\oh)=132\cdot 143\cdot 124,\quad \pi^*(g^2hh'\oh)=132\cdot 143\cdot 124,
$$
$$
\pi^*(ghh'^2\oh)=1342\cdot 1423\cdot 1234,\quad \pi^*(g^2h\oh'^2\oh)=1243\cdot 1324\cdot 1432.
$$
We note that although the three lifts of the prime $ghh'\oh$ are the prime paths $123$, $134$, and $142$, the fact that $F(ghh'\oh)=g^2$ tells us that $f(ghh'\oh)=3$ and hence that $ghh'\oh$ has a unique prime of length nine over it, namely $123134142$ (see Rem.~\ref{rem:shortestclosed}).

\end{example}

\begin{example} \label{ex:K4modA4}
We also consider the quotient $X=K_4/A_4$, which consists of a single vertex $u$ with a leg $l$. Let $T=\{1\}$ and let $S_1=\{1,h_{12}\}$ be the identity sheet, so that $S_x=\{x(1),h_{x(1)x(2)}\}$ for any $x\in A_4$. The stabilizer at $u$ is generated by $g=(234)$, and $F(l)=(12)(34)$, which is the unique element of $A_4$ mapping $h_{21}$ to $h_{12}$. 

The quotient $\XX$ has two primes of length one, $gl$ and $g^2l$. The Frobenius elements $F(gl)=(132)$ and $F(g^2l)=(142)$ have order 3, therefore there are four primes of $K_4$ of length three over each of them, which accounts for all primes $abc$ of $K_4$ of length three. There is a unique prime of length two, namely $glg^2l$. We see that $F(glg^2l)=(14)(23)$ has order 2, so the primes in $K_4$ over $glg^2l$ are the six primes $abcd$ of length four, which are all conjugate by the $A_4$-action. There are no remaining primes of $K_4$ of length less than six, and a laborious calculation shows that all primes of $\XX$ of lengths three, four, and five indeed have nontrivial Frobenius elements.

\end{example}

\subsection{The $L$-function of an edge-free quotient}

In this section, we define the $L$-function of an edge-free quotient, in complete analogy to free Galois coverings.

\begin{definition} Let $Y$ be a finite graph with an edge-free $G$-action, let $Y/\!\!/G=\XX$ be the quotient graph of groups, and let $\rho$ be a representation of $G$ of degree $d$. The {\it Artin--Ihara $L$-function} of the covering $Y/\XX$ is
$$
L(u,\rho,Y/\XX)=\prod_{\fq} \det\left(I_d-\rho(F(Q))u^{\ell(Q)}\right)^{-1},
$$
where the product is taken over all primes $\fq$ of $\XX$, for each of which we choose an arbitrary representative $Q$.
\end{definition}

Jus as the zeta function of a graph, this product is usually infinite but is in fact the reciprocal of a polynomial in $u$ (see Thm.~\ref{thm:threetermgogL}). We note that if $Q$ and $Q'$ are representatives of $\fq$, then the Frobenius elements $F(Q)$ and $F(Q')$ are conjugate, hence the determinant in the product is well-defined.

As with the $L$-function of a free Galois covering, it is clear that
$$
L(u,\rho_1\oplus \rho_2,Y/\XX)=L(u,\rho_1,Y/\XX)L(u,\rho,Y/\XX)
$$
for any two representations $\rho_1$ and $\rho_2$, and that 
$$
L(u,1_G,Y/\XX)=\zeta(u,\XX)
$$
for the trivial representation $1_G$.

Our principal result concerning the $L$-function of an edge-free covering is the generalization of~\eqref{eq:Linduction}.

\begin{proposition} Let $Y$ be a finite graph with an edge-free $G$-action, let $Y/\!\!/G=\XX$ be the quotient graph of groups, and let $\rho_G$ be the right regular representation of $G$. Then
$$
\zeta(u,Y)=L(u,\rho_G,Y/\XX).
$$

\end{proposition}

\begin{proof} The proof is based on the relationship between the primes of $Y$ and $\XX$ established in Thm.~\ref{thm:primes}. Since this relationship is the same as in the free case, the proof is essentially identical (see part 3 of Prop.~18.10 and Sec.~19.3 in~\cite{2011Terras}, specialized to $H=1$).  Let $d=\#(G)$ be the degree of $\rho_G$. Consider the logarithm
$$
\log L(u,\rho_G,Y/\XX)=-\sum_{\fq}\log \det \left(I_d-\rho_G(F(Q))u^{\ell(Q)}\right)=\sum_{\fq} \sum_{j=1}^{\infty} \frac{1}{j} \tr \rho_G(F(Q^j)) u^{j\ell(Q)},
$$
where we have chosen a representative $Q$ for each prime $\fq$ of $\XX$. Since $G$ acts freely on itself, $\tr \rho_G(F(Q^j))=0$ unless $F(Q^j)=1$, in which case $\tr \rho_G(F(Q^j))=d$. But $F(Q^j)=1$ if and only if $j=kf(Q)$, where $f(Q)$ is the order of the Frobenius element of $Q$. Hence
$$
\log L(u,\rho_G,Y/\XX)=\sum_{\fq} \sum_{k=1}^{\infty} \frac{d}{kf(Q)} u^{kf(Q)\ell(Q)}.
$$
By Thm.~\ref{thm:primes}, there are $g(Q)=d/f(Q)$ primitive paths $P_1,\ldots,P_g$ in $Y$ over $Q$, each having length $\ell(P_i)=f(Q)\ell(Q)$. Furthermore, each prime $\fp$ lying over $\fq$ is represented by exactly one of the $P_i$. Hence we can change to summing over all primes $\fp$ of $Y$:
$$
\log L(u,\rho_G,Y/\XX)=\sum_{\fq}\sum_{k=1}^{\infty} \sum_{i=1}^{g(Q)}\frac{1}{k} u^{k\ell(P_i)}=\sum_{\fp}\sum_{k=1}^{\infty} \frac{1}{k} u^{k\ell(\fp)},
$$
which, as in the proof of Prop.~\ref{prop:twoterm}, is the logarithm of $\zeta(u,Y)$.

\end{proof}

We finally obtain the divisibility result for the zeta function of an edge-free quotient.

\begin{corollary}  Let $Y$ be a finite graph with an edge-free $G$-action and let $\XX=Y/\!\!/G$ be the quotient graph of groups. Then the zeta function $\zeta(u,\XX)$ divides the zeta function $\zeta(u,Y)$.

\label{cor:divisibility}
\end{corollary}

\begin{proof} This follows immediately from the above theorem. Indeed, let $R_G$ be the set of irreducible representations of $G$. The regular representation $\rho_G$ splits, and therefore $\zeta(y,Y)$ factors, as follows:
$$
\rho_G=\bigoplus_{\rho\in R_G} d_{\rho}\rho,\quad \zeta(y,Y)=L(u,\rho_G,Y/\XX)=\prod_{\rho\in R_G}L(u,\rho,Y/\XX)^{d_{\rho}},
$$
and $\zeta(u,\XX)$ is the factor corresponding to the trivial representation.

\end{proof}

We now derive two-term and three-term determinant formulas for the $L$-function of an edge-free quotient, generalizing the corresponding results for free Galois coverings (see Chapters 18 and 19 of~\cite{2011Terras}). Let $Y/\XX$ be an edge-free quotient with Galois group $G$, and let $\rho$ be a representation of $G$ of degree $d$. Let $n=\#(V(X))$ and $k=\#(H(X))$ denote respectively the number of vertices and half-edges of $X$. We define the {\it Artinized charge} of a vertex $v\in V(X)$ to be the $d\times d$ matrix
$$
c_{\rho}(v)=\sum_{g\in X_v}\rho(g),
$$
and we consider the $kd\times kd$ block matrix $W_{\rho}$ whose $(h,h')$ block is
$$
(W_{\rho})_{hh'}=\left\{\begin{array}{cc}
(c_{\rho}(r(h'))-I_d)\rho(F(h')), & h'=\oh,\\
c_{\rho}(r(h'))\rho(F(h')), & h'\neq \oh\mbox{ and }r(\oh)=r(h'), \\
0, & r(\oh)\neq r(h').
\end{array}\right.
$$

\begin{proposition} Let $Y$ be a finite graph with an edge-free $G$-action, let $\XX=Y/\!\!/G$ be the quotient graph of groups, and let $\rho_G$ be a representation of $G$ of degree $d$. The $L$-function of $Y/\XX$ evaluated at $\rho$ is equal to
\label{prop:2termL}
$$
L(u,\rho,Y/\XX)^{-1}=\det (I_{kd}-W_{\rho}u).
$$

\end{proposition}

\begin{proof} The proof is similar to that of Prop.~\ref{prop:twoterm}. We take the logarithm of the $L$-function and change to a sum over all primitive paths $Q$, and then to a sum over all reduced paths $C$:
$$
\log L(u,\rho,Y/\XX)=-\sum_{\fq} \log \det \left(I_d-\rho(F(Q))u^{\ell(Q)}\right)=-\sum_{\fq} \tr \log \left(I_d-\rho(F(Q))u^{\ell(Q)}\right)=
$$
$$
=\sum_{\fq}\sum_{j=1}^{\infty}\frac{1}{j}\tr \rho (F(Q))^ju^{j\ell(Q)}=\sum_Q\sum_{j=1}^{\infty}\frac{1}{j\ell(Q)}\tr \rho (F(Q))^ju^{j\ell(Q)}=\sum_C\frac{1}{\ell(C)}\tr \rho (F(C))u^{\ell(C)}.
$$
Fix a closed, but not necessarily reduced, path $C_X=h_1h_2\cdots h_n$ in $X$, and consider all reduced paths $C=g_0h_1g_1h_2\cdots g_{n-1}h_n$ in $\XX$ whose underlying half-edge path is $C_X$. These paths are obtained by choosing $g_j\in X'_{r(\oh_j)}$ arbitrarily, where $X'_{r(\oh_j)}=X_{r(\oh_j)}\backslash\{1\}$ if $h_{j+1}=\oh_j$ and $X'_{r(\oh_j)}=X_{r(\oh_j)}$ otherwise. It follows that the sum over all paths $C$ with underlying half-edge path $C_X$ is equal to
$$
\sum\frac{1}{n}\tr \rho (F(C))u^n=\sum_{g_j\in X'_{r(\oh_j)}}\frac{1}{n}\tr (\rho(g_0)\rho (F(h_1)) \rho(g_1)\cdots \rho(g_{n-1})\rho (F(h_n)))=
$$
$$
=\frac{1}{n} \tr ((W_{\rho})_{h_1h_2}(W_{\rho})_{h_2h_3}\cdots (W_{\rho})_{h_nh_1}),
$$
where each $(W_{\rho})_{h_jh_{j+1}}$ is a $d\times d$ block of the matrix $W_{\rho}$. If we now sum this expression over all closed half-edge paths $C_X$ in $X$, we obtain $\dfrac{1}{n} \tr (W_{\rho})^n$. Therefore 
$$
\log L(u,\rho,Y/\XX)=\sum_{C:\ell(C)=n}\frac{1}{n} \tr \rho (F(C))u^n=\sum_{n=1}^{\infty}\frac{1}{n} \tr(W_{\rho})^nu^n=\tr \log (I_{kd}-W_{\rho}u)^{-1},
$$
and exponentiating completes the proof.

\end{proof}

Finally, we derive a three-term determinant formula for $L(u,\rho,Y/\XX)$. Define the following $nd\times nd$ matrices:
\begin{enumerate}

\item $A_{\rho}$ is the {\it Artinized adjacency matrix,} with $d\times d$ blocks of the form
$$
(A_{\rho})_{uv}=\sum_{r(h)=u,r(\oh)=v} \rho (F(h)).
$$

\item $Q_{\rho}=Q\otimes I_d$ is the {\it Artinized valency matrix}, where $Q$ is the valency matrix of $X$. 

\item $(C_\rho)_{uv}=\delta_{uv}c_{\rho}(v)$ is the {\it Artinized charge matrix}. 

\end{enumerate}

We also consider the $d\times d$ matrices $\rho(F(l))$ for all legs $l\in L(X)$. Since $F(l)^2=1$, these matrices are diagonalizable with eigenvalues $\pm 1$, and we denote by respectively $l_{\rho,+}$ and $l_{\rho,-}$ the total number of $+1$ and $-1$ eigenvalues of all the $\rho(F(l))$, so that $l_{\rho,+}+l_{\rho,-}=ld$.

\begin{theorem} Let $Y$ be a finite graph with an edge-free $G$-action, let $\XX=Y/\!\!/G$ be the quotient graph of groups, and let $\rho_G$ be a representation of $G$ of degree $d$.  The $L$-function of $Y/\XX$ evaluated at $\rho$ is equal to
\begin{equation}
L(u,\rho,Y/\XX)^{-1}=(1-u^2)^{(b_1(X)-1)d}(1+u)^{l_{\rho,+}}(1-u)^{l_{\rho,-}}
\det(I_{nd}-C_{\rho}A_{\rho}u+(C_{\rho}Q_{\rho}-I_{nd})u^2),
\end{equation}
where $b_1(X)=m-n+1$ is the first Betti number of $X$.
\label{thm:threetermgogL}

\end{theorem}

\begin{proof} Our proof closely follows the proof of the corresponding result for $L$-functions of free coverings (see pp.~170-173 in~\cite{2011Terras}), and consists in introducing a block-diagonal matrix that records the vertex stabilizers.

We recall the auxiliary matrices $S$, $T$, and $J$ from the proof of Prop.~\ref{thm:3term}, and introduce the corresponding {\it Artinized matrices}
$$
S_{\rho}=S\otimes I_d,\quad T_{\rho}=T\otimes I_d,\quad J_{\rho}=J\otimes I_d.
$$
We also consider the block-diagonal $kd\times kd$ matrices $D_{\rho}$ and $R_{\rho}$:
$$
(D_{\rho})_{hh'}=\delta_{hh'}c_{\rho}(h'),\quad (R_{\rho})_{hh'}=\delta_{hh'}\rho (F(h')).
$$
It is readily verified that these matrices satisfy the following properties:
$$
S_{\rho}J_{\rho}=T_{\rho},\quad T_{\rho}J_{\rho}=S_{\rho},\quad {}^tJ_{\rho}=J_{\rho},\quad W_{\rho}+R_{\rho} J_{\rho}=R_{\rho}{}^tT_{\rho}S_{\rho}D_{\rho},\quad (R_{\rho}J_{\rho})^2=I_{kd},
$$
$$
S_{\rho}D_{\rho}=C_{\rho}S_{\rho},\quad Q_{\rho}=S_{\rho} {}^t S_{\rho}, \quad A_{\rho}=S_{\rho}R_{\rho}{}^tT_{\rho}.
$$
Therefore we have
$$
C_{\rho}Q_{\rho}=S_{\rho}D_{\rho}{}^tS_{\rho},\quad C_{\rho}A_{\rho}=S_{\rho}D_{\rho}R_{\rho}{}^tT_{\rho}.
$$
A direct calculation verifies the following matrix identity:
$$
\left(\begin{array}{cc} I_{nd} & 0 \\ R_{\rho}{}^tT_{\rho} & I_{kd} \end{array}\right)
\left(\begin{array}{cc} I_{nd}(1-u^2) & S_{\rho}D_{\rho}u \\ 0 & I_{kd}-W_{\rho}u \end{array}\right)=
$$
$$
=\left(\begin{array}{cc} I_{nd}-S_{\rho}D_{\rho}R_{\rho}{}^tT_{\rho}u+(S_{\rho}D_{\rho}{}^tS_{\rho}-I_{nd})u^2 & S_{\rho}D_{\rho}u \\ 0 & I_{kd}+R_{\rho}J_{\rho} u \end{array}\right)
\left(\begin{array}{cc} I_{nd} & 0 \\ R_{\rho} {}^tT_{\rho} -{}^tS_{\rho} u & I_{kd} \end{array}\right).
$$
Taking determinants and using Prop.~\ref{prop:2termL}, we see that
$$
(1-u^2)^{nd}L(u,\rho,Y/\XX)^{-1}=\det( I_{nd}-S_{\rho}D_{\rho}R_{\rho}{}^tT_{\rho}u+(S_{\rho}D_{\rho}{}^tS_{\rho}-I_{nd})u^2)\det(I_{kd}+R_{\rho}J_{\rho}u).
$$
To complete the proof, we need to compute the determinant $\det(I_{kd}+R_{\rho}J_{\rho}u)$. Ordering the $k=2m+l$ half-edges of $X$ as in the proof of Thm.~\ref{thm:3term}, the matrices $J_{\rho}$ and $R_{\rho}$ have the following form:
$$
J_{\rho}=\left(\begin{array}{ccc} 0 & I_{md} & 0 \\ I_{md} & 0 & 0 \\ 0 & 0 & I_{ld} \end{array}\right),\quad
R_{\rho}=\left(\begin{array}{ccc} U & 0 & 0 \\ 0 & U^{-1} & 0 \\ 0 & 0 & V \end{array}\right).
$$
Therefore
$$
\det(I_{kd}+R_{\rho}J_{\rho}u)=\left|\begin{array}{ccc} I_{md} & Uu & 0 \\ U^{-1}u & I_{md} & 0 \\ 0 & 0 & I_{ld}+Vu \end{array}\right|=\left|\begin{array}{cc} I_{md} & Uu \\ U^{-1}u & I_{md} \end{array}\right|\det(I_{ld}+Vu).
$$
The first determinant can be computed using the identity
$$
\left(\begin{array}{cc} I_{md} & 0 \\ -U^{-1}u & I_{md} \end{array}\right)
\left(\begin{array}{cc} I_{md} & Uu \\ U^{-1}u & I_{md} \end{array}\right)=
\left(\begin{array}{cc} I_{md} & Uu \\ 0 & I_{md}(1-u^2) \end{array}\right)
$$
and is equal to
$$
\left|\begin{array}{cc} I_{md} & Uu \\ U^{-1}u & I_{md} \end{array}\right|=(1-u^2)^{md}.
$$
For the second determinant, we note that $\det(I_{ld}u-V)$ is the product of the characteristic polynomials of the $\rho(F(l))$ over all $l\in L(X)$ and so is equal to $(u-1)^{l_{\rho,+}}(u+1)^{l_{\rho,-}}$. Therefore,
$$
\det(I_{ld}+Vu)=\det\left(-u\left(-\frac{1}{u}I_{ld}-V\right)\right)=(-u)^{ld}\left(-\frac{1}{u}-1\right)^{l_{\rho,+}}\left(-\frac{1}{u}+1\right)^{l_{\rho,-}}=(1+u)^{l_{\rho,+}}(1-u)^{l_{\rho,-}}.
$$
Putting this together, and replacing all auxiliary matrices with $A_{\rho}$, $Q_{\rho}$, and $C_{\rho}$, we obtain the proof.

\end{proof}

To illustrate our results, we recall the zeta function of the tetrahedron graph $K_4$:
$$
\zeta(u,K_4)^{-1}=(1-u)^3(1+u)^2(1-2u)(1+u+2u^2)^3.
$$
We calculate the $L$-functions of the two quotients of $K_4$ that we considered in Ex.~\ref{ex:K4modC3} and Ex.~\ref{ex:K4modA4}, and observe how in each case $\zeta(u,K_4)$ factors as a product of the $L$-functions over the irreducible representations of the Galois group.

\begin{example} Consider the quotient $\XX=K_4/C_3$ of Ex.~\ref{ex:K4modC3}. The group $C_3=\{1,g,g^2\}\subset A_4$, where $g=(234)$, has two nontrivial one-dimensional representations $\rho$ and $\rho^2$, defined by $\rho(g)=\zeta$, where $\zeta$ is a primitive cube root of unity. In the notation of Ex.~\ref{ex:K4modC3}, we find that
$$
\rho(F(h))=\rho(F(\oh))=1,\quad \rho(F(h'))=\zeta,\quad \rho(F(\oh'))=\zeta^2, \quad c_{\rho}(u)=1+\zeta+\zeta^2=0,\quad c_{\rho}(v)=1,
$$
hence
$$
A_{\rho}=\left(\begin{array}{cc} 0 & 1 \\ 1 & -1\end{array}\right),\quad
Q_{\rho}=\left(\begin{array}{cc} 1 & 0 \\ 0 & 3 \end{array}\right),\quad
C_{\rho}=\left(\begin{array}{cc} 0 & 0 \\ 0 & 1 \end{array}\right).
$$
In addition, $b_1(X)=1$, and $l_{\rho,+}=l_{\rho,-}=0$ since there are no legs. The representation $\rho^2$ is the complex conjugate of $\rho$, hence it has the same matrices $A$, $Q$, and $C$. Therefore
$$
L(u,\rho,K_4/\XX)^{-1}=L(u,\rho^2,K_4/\XX)^{-1}=\left|\begin{array}{cc} 1-u^2 & 0 \\ -u & 1+u+2u^2 \end{array}\right|=(1-u)(1+u)(1+u+2u^2).
$$
We saw in Ex.~\ref{ex:K4} that $\zeta(u,\XX)^{-1}=(1-u)(1-2u)(1+u+2u^2)$, therefore
$$
\zeta(u,K_4)=\zeta(u,\XX)L(u,\rho,K_4/\XX)L(u,\rho^2,K_4/\XX),
$$
in accordance with the splitting $\rho_{C_3}=1\oplus \rho\oplus \rho^2$ of the regular representation of $C_3$.

\end{example}

\begin{example} We also consider the quotient $\XX=K_4/A_4$ of Ex.~\ref{ex:K4modA4}. The group $A_4$ has two nontrivial one-dimensional representations $\rho$ and $\rho^2$, defined by $\rho(g)=\zeta$, where $g=(234)$ and $\zeta$ is a primitive cube root of unity, and that are trivial on the Klein $4$-group $V_4\subset A_4$. Therefore
$$
\rho(F(l))=\rho((12)(34))=1,\quad c_{\rho}=1+\zeta+\zeta^2=0,\quad C_{\rho}=(0),\quad l_{\rho,+}=1,\quad l_{\rho,-}=0,
$$
hence
$$
L(u,\rho,K_4/\XX)^{-1}=L(u,\rho^2,K_4/\XX)^{-1}=1+u.
$$
Finally, let $\sigma$ be the irreducible three-dimensional representation of $A_4$, with basis $y_1=e_1-e_2$, $y_2=e_1-e_3$, $y_3=e_1-e_4$. We compute the matrices $A_{\sigma}$, $Q_{\sigma}$, and $C_{\sigma}$ in this basis:
$$
A_{\sigma}=\sigma (F(l))=\sigma ((12)(34))=\left(\begin{array}{ccc} -1 & -1 & -1 \\ 0 & 0 & 1 \\ 0 & 1 & 0\end{array}\right),\quad Q_{\sigma}=I_3,
$$
$$
C_{\sigma}=c_{\sigma}(u)=I_3+c_{\sigma}((234))+c_{\sigma}((243))=\left(\begin{array}{ccc} 1 & 1 & 1 \\ 1 & 1 & 1 \\ 1 & 1 & 1\end{array}\right).
$$
The matrix $\sigma(F(l))$ has one eigenvalue $+1$ and two eigenvalues $-1$, therefore $l_{\sigma,+}=1$ and $l_{\sigma,-}=2$, and $b_1(X)=0$. Hence we calculate that
$$
L(u,\sigma,K_4/\XX)^{-1}=(1-u^2)^{-3}(1+u)(1-u)^2\det(I_3-C_{\sigma}A_{\sigma}u+(C_{\sigma}Q_{\sigma}-I_3)u^2)=(1-u)(1+u+2u^2).
$$
Since $\zeta(u,\XX)^{-1}=1-2u$, we observe that
$$
\zeta(u,K_4)=\zeta(u,\XX)L(u,\rho,K_4/\XX)L(u,\rho^2,K_4/\XX)L(u,\sigma,K_4/\XX)^3,
$$
which agrees with the splitting $\rho_{A_4}=1\oplus \rho\oplus \rho^2\oplus 3\sigma$ of the regular representation of $A_4$.

\end{example}

\section{Acknowledgments}

The author would like to thank Scott Corry, Yoav Len and Martin Ulirsch for thoughtful discussions.



\end{document}